\documentclass[12pt]{amsart}

\usepackage{amssymb,verbatim,amscd,amsmath,graphicx,enumerate}
\usepackage[T1]{fontenc}
\usepackage{graphicx}
\usepackage{caption}
\usepackage{subcaption}
\usepackage{pdfsync}
\usepackage{fullpage}
\usepackage{color}
\usepackage{marginnote}
\usepackage{tikz}
\usepackage{bbm}
\usepackage{fancybox}

\numberwithin{equation}{section}

{\theoremstyle{plain}
    \newtheorem{thm}{\bf Theorem}[section]
    \newtheorem{proposition}[thm]{\bf Proposition}
    \newtheorem{corollary}[thm]{\bf Corollary}
    
}
{\theoremstyle{definition}
    \newtheorem{defn}[thm]{\bf Definition}
    \newtheorem{conjecture}[thm]{\bf Conjecture}
    \newtheorem{problem}[thm]{\bf Problem}
    \newtheorem{question}[thm]{\bf Question}

    \newtheorem{example}[thm]{\bf Example}
    \newtheorem*{ack}{\bf Acknowledgement}
}

\DeclareMathOperator{\reg}{reg}
\DeclareMathOperator{\depth}{depth}

\DeclareMathOperator{\pd}{pd}

\DeclareMathOperator{\pdrSpec}{pdreg}
\newcommand{\mc}{\tau_{\max}}

\newcommand{\ZZ}{{\mathbb Z}}
\newcommand{\NN}{{\mathbb N}}

\def\mm{{\frak m}}

\def\1{{\bf 1}}
\def\0{{\bf 0}}

\begin{document}

\title{MAX MIN vertex cover and the size of Betti tables}

\author[H.T. H\`a]{Huy T\`ai H\`a}
\address{Department of Mathematics, Tulane University, 6823 St. Charles Avenue, New Orleans, LA 70118, USA}
\email{tha@tulane.edu}

\author[T. Hibi]{Takayuki Hibi}
\address{Department of Pure and Applied Mathematics, Graduate School of Information Science and Technology, Osaka University, Suita, Osaka 565-0871, Japan}
\email{hibi@math.sci.osaka-u.ac.jp}


\begin{abstract}
Let $G$ be a finite simple graph on $n$ vertices, that contains no isolated vertices, and let $I(G) \subseteq S = K[x_1, \dots, x_n]$ be its edge ideal. In this paper, we study the pair of integers that measure the projective dimension and the regularity of $S/I(G)$. We show that if $\pd(S/I(G))$ attains its minimum possible value $2\sqrt{n}-2$ then, with only one exception, $\reg(S/I(G)) = 1$. We also provide a full description of the spectrum of $\pd(S/I(G))$ when $\reg(S/I(G))$ attains its minimum possible value 1.
\end{abstract}

\maketitle

\section{Introduction}
In the current trends of commutative algebra, the role of combinatorics is distinguished.  Particularly, the combinatorics of finite graphs has created fascinating research problems in commutative algebra and, vice-versa, algebraic methods and techniques have shed new lights on graph-theoretic questions (\cite[Chapters 9 and 10]{HHgtm260}).

Let $G$ be a simple graph over the vertex set $V_G = \{x_1, \dots, x_n\}$ and edge set $E_G$. Throughout the paper, all our graphs are assumed to contain no isolated vertices.  Let $K$ be a field and identify the vertices in $V_G$ with the variables in the polynomial ring $S = K[x_1, \ldots, x_n]$.  The \emph{edge ideal} of $G$, first introduced by Villarreal \cite{V}, is defined by
$$I(G) = \langle x_i x_j ~\big|~ \{x_i, x_j\} \in E_G\rangle \subseteq S.$$
Let $\pd(G)$ and $\reg(G)$ denote the \emph{projective dimension} and the Castelnuovo--Mumford \emph{regularity} of $S/I(G)$, respectively.  These are fundamental homological invariants that measure the computational complexity of $S/I(G)$. Particularly, $\pd(G)$ and $\reg(G)$ describe the size of the \emph{graded Betti table} of $S/I(G)$. Our work in this paper is motivated by the following basic question.

\begin{question} \label{question}
	Given a positive integer $n$, for which pairs of integers $(p,r)$, there exists a graph $G$ on $n$ vertices such that $\pd(G) = p$ and $\reg(G) = r$?
\end{question}

Our approach to Question \ref{question} is purely combinatorial in nature. Specifically, we investigate an important graph-theoretic invariant, namely, the maximum size of a minimal vertex cover of $G$, which we shall denote by $\mc(G)$. In graph theory, the two symmetric dual problems, to find a \textsc{max min vertex cover} and to find a \textsc{min max independent set} in a graph, are known to be \textbf{NP}-hard problems and, in recent years, have received a growing attention \cite{BdCP, BdCEP, CHLN, GKL, GL, Hall}. Particular, a result of Costa, Haeusler, Laber and Nogueira \cite[Theorem 2.2]{CHLN} proves that
\begin{align}
\mc(G) \ge 2\sqrt{n}-2. \label{eq.tau}
\end{align}
The inequality (\ref{eq.tau}) gives rise to a somewhat surprising bound, that may have been overlooked, that
$\pd(G) \ge 2\sqrt{n}-2$; see Corollary \ref{cor.pd}.

Our first main result shows that if $\pd(G) = 2\sqrt{n}-2$ then, with only one exception, we must have $\reg(G) = 1$. In fact, since $\pd(G) \ge \mc(G)$ (see, for example, the proof of Corollary \ref{cor.pd}), we prove the following stronger statement. For a positive integer $s$, let $H_s$ denote the graph consisting of a complete subgraph $K_s$ each of whose vertex is connected to a set of $s-1$ independent vertices, and these independent sets are pairwise disjoint (see Figure \ref{figHs}). Let $2K_2$ denote the graph consisting of two disjoint edges and let $C_4$ be the induced $4$-cycle.

\medskip

\noindent\textbf{Theorem \ref{thm.reg}.} Let $G$ be a simple graph on $n$ vertices and suppose that $n$ is a perfect square. If $\mc(G) = 2\sqrt{n}-2$ then we must have either
\begin{enumerate}
	\item $G = 2K_2$ and $(\pd(G), \reg(G)) = (2,2)$, or
	\item $G = C_4$ and $(\pd(G), \reg(G)) = (3,1)$, or
	\item $G = H_s$, for some $s \in \NN$, and $(\pd(G),\reg(G)) = (2\sqrt{n}-2,1)$.
\end{enumerate}

To establish Theorem \ref{thm.reg}, we characterize all graphs $G$ for which $\mc(G) = 2\sqrt{n}-2$, when $n$ is a perfect square. Our classification reads as follows.

\medskip

\noindent\textbf{Theorem \ref{classification}.} Let $G$ be a simple graph on $n$ vertices and suppose that $n$ is a perfect square. Then $\mc(G) = 2\sqrt{n}-2$ if and only if $G$ is either $2K_2$, $C_4$ or $H_s$, for some $s \in \NN$.

\medskip

Theorem \ref{thm.reg} further leads to the following natural special case of Question \ref{question}: what is the spectrum of $\pd(G)$ for all graphs $G$, for which $\reg(G) = 1$? Our next main result answers this question.

\medskip

\noindent\textbf{Theorem \ref{thm.spec}.} Let $n \ge 2$ be any integer. The spectrum of $\pd(G)$ for all graphs $G$, for which $\reg(G) = 1$, is precisely $[2\sqrt{n}-2, n-1] \cap \ZZ$.

\medskip

As an immediate consequence of Theorem \ref{thm.spec}, we show in Corollary \ref{cor.spec} that for any pairs of positive integers $(p,r)$ such that $r \le n/2$ and
$$2 \sqrt{n-2(r-1)}+r-3 \le p \le n-r,$$
there exists a graph on $n$ vertices such that $(\pd(G),\reg(G)) = (p,r)$.

The paper is outlined as follows. Section \ref{sec.graph} collects basic notations and terminology of finite simple graphs that will be used in the paper. In Section \ref{sec.gapfree}, we start with a simple proof for the inequality (\ref{eq.tau}) when $G$ is a gap-free graph, giving the non-experts (in our case, the algebraic readers) a glimpse of why $2\sqrt{n}-2$ appears naturally in the bound for $\mc(G)$; see Theorem \ref{gap-free}. We also construct, for any given $n \in \NN$, a gap-free and chordal graph $G$ such that $\mc(G) = \lceil 2\sqrt{n}-2\rceil$, exhibiting that the inequality (\ref{eq.tau}) is sharp. Section \ref{sec.main} is devoted to a new proof of the inequality (\ref{eq.tau}) in the most general situation; see Theorem \ref{tulane}. Our proof is different from that given in \cite[Theorem 2.2]{CHLN} and provides structures that we can later on use in the classification for graphs where the equality is attained. Section \ref{sec.class} contains our main results of the paper. We classify graphs $G$ for which $\mc(G) = 2\sqrt{n}-2$, in Theorem \ref{classification}, and examine $(\pd(G), \reg(G))$ in this case, in Theorem \ref{thm.reg}. We finally show that when $\reg(G) = 1$, the spectrum of $\pd(G)$ is $[2\sqrt{n}-2, n-1] \cap \ZZ$, in Theorem~\ref{thm.spec}.

\begin{ack} The first author was supported by Louisiana BOR grant LEQSF(2017-19)-ENH-TR-25, and the second author was supported by JSPS KAKENHI 19H00637. The authors would like to thank Martin Milanic for informing us that the inequality (\ref{eq.tau}) was already known and drawing our attention to \cite{CHLN}. The authors would also like to thank Antonio Macchia for pointing out to us that $C_4$ should be included in our classification result.
\end{ack}


\section{Combinatorics of finite graphs and algebraic invariants} \label{sec.graph}
Throughout the paper $G$ shall denote a finite simple graph on $n \ge 2$ vertices that contains no isolated vertices. Recall that a finite graph $G$ is \emph{simple} if $G$ has no \emph{loops} nor \emph{multiple edges}. We shall use $V_G$ and $E_G$ to denote the vertex and edge sets of $G$, respectively.

\begin{defn} Let $G$ be a graph.
\begin{enumerate}
	\item A subset $W \subseteq V_G$ is called a \emph{vertex cover} of $G$ if $e \cap W \not= \emptyset$ for every edge $e \in E_G$. A vertex cover $W$ is \emph{minimal} if no proper subset of $W$ is also a vertex cover of $G$.  Set
	$$\mc(G) = \max\{ |W| ~\big|~ W \text{ is a minimal vertex cover of } G\}.$$
	\item A subset $W \subseteq V_G$ is called an \emph{independent set} if for any $u\not=v \in W$, $\{u,v\} \not\in E_G$.
\end{enumerate}
\end{defn}
\noindent It is easy to see that the complement of a vertex cover is an independent set. Particularly, the complement of a \textsc{maximum minimal vertex cover} is a \textsc{minimum maximal independent set}.

For a subset $W \subseteq V_G$, the \emph{induced subgraph} of $G$ over $W$ is the graph whose vertex set is $W$ and whose edge set is $\{\{u,v\} ~\big|~ u,v \in W \text{ and } \{u,v\} \in E_G\}$. A subset $M$ of $E_G$ is a \emph{matching} of $G$ if, for any $e \not= e'$ in $M$, one has $e \cap e' = \emptyset$.  The \emph{matching number} of $G$ is the largest size of a matching in $G$, and is denoted by $\beta(G)$. A matching $M$ of $G$ is called an \emph{induced matching} of $G$ if the induced subgraph of $G$ over $\bigcup_{e \in M} e$ has no edges other than those already in $M$.  The \emph{induced matching number} of $G$ is the largest size of an induced matching $G$, and is denoted by $\nu(G)$.  It is known from \cite{HVT, Ka} that $\nu(G) \leq \reg(G) \leq \beta(G)$.

\begin{defn} Let $G$ be a graph.
\begin{enumerate}
	\item $G$ is called \emph{gap-free} if $\nu(G) = 1$. Equivalently, $G$ is gap-free if for any two disjoint edges $e, f \in E_G$, there exists an edge $g \in E_G$ such that $e \cap g \not= \emptyset$ and $f \cap g \not=\emptyset$.
	\item $G$ is called \emph{chordal} if every cycle of length at least 4 in $G$ has a \emph{chord}. That is, for every cycle $C$ of length at least 4 in $G$, there exist two nonconsecutive vertices $u$ and $v$ on $C$ such that $\{u,v\} \in E_G$.
\end{enumerate}
\end{defn}
\noindent It is known from \cite{DS, HVT} that if $G$ is a chordal graph then $\pd(G) = \mc(G)$ and $\reg(G) = \nu(G)$.

\begin{defn}
	Let $W \subseteq V_G$ be a subset of the vertices in $G$. The \emph{neighborhood} (the set of \emph{neighbors}) and the \emph{closed neighborhood} of $W$ are defined by
	$$N_G(W) = \{u \in V_G ~\big|~ \exists w \in W: \{u,w\} \in E_G\} \text{ and } N_G[W] = N_G(W) \cup W.$$
\end{defn}
\noindent When $W = \{v\}$, for simplicity of notation, we shall write $N_G(v)$ and $N_G[v]$ in place of $N_G(W)$ and $N_G[W]$. The \emph{degree} of a vertex $v \in V_G$ is defined to be $\deg_G(v) = |N_G(v)|$. A vertex $v \in V_G$ is called \emph{free} if $\deg_G(v) = 1$. A \emph{leaf} in $G$ is an edge that contains a free vertex.

A \emph{path} is an alternating sequence of distinct vertices and edges (except possibly the first and the last vertex) $x_1, e_1, x_2, e_2, \dots, e_s, x_{s+1}$ such that $e_i = \{x_i, x_{i+1}\}$, for $i = 1, \dots, s$. The \emph{length} of a path is the number of edges on the path. A path of length $s$ is denoted by $P_s$. A \emph{cycle} is a closed path. An \emph{induced cycle} in $G$ is a cycle which is also an induced subgraph of $G$. An induced cycle of length $s$ is denoted by $C_s$. A \emph{complete} graph is a graph in which any two distinct vertices are connected by an edge. We shall use $K_s$ to denote a  complete graph over $s$ vertices. A graph $G$ is \emph{bipartite} if there is a partition $V_G = X \cup Y$ of the vertices of $G$ such that every edge in $G$ connects a vertex in $X$ to a vertex in $Y$. A bipartite graph $G$ with a bi-partition $V_G = X \cup Y$ of its vertices is called a \emph{complete bipartite} graph if $E_G = \{\{x,y\} ~\big|~ x\in X, y\in Y\}$. We use $K_{r,s}$ to denote the complete bipartite graph whose vertices are partitioned into the union of two sets of cardinality $r$ and $s$.

Finally, let $K$ be a field and let $S = K[V_G]$ represent the polynomial ring associated to the vertices in $G$.

\begin{defn}
	Let $M$ be a finitely generated graded $S$ module. Then $M$ admits a \emph{minimal graded free resolution} of the form
	$$0 \rightarrow \bigoplus_{j \in \ZZ}S(-j)^{\beta_{p,j}(M)} \rightarrow \dots \rightarrow \bigoplus_{j\in \ZZ}S(-j)^{\beta_{0,j}(M)} \rightarrow M \rightarrow 0.$$
	The number $\beta_{i,j}(M)$ are called the \emph{graded Betti numbers} of $M$. The \emph{projective dimension} and the \emph{regularity} of $M$ are defined as follows:
	$$\pd(M) = \max\{i ~\big|~ \exists j: \beta_{i,j}(M) \not= 0\} \text{ and } \reg(M) = \max\{j-i ~\big|~ \beta_{i,j}(M) \not= 0\}.$$
    The \emph{graded Betti table} of $M$ is an $\pd(M) \times \reg(M)$ array whose $(i,j)$-entry is $\beta_{i,i+j}(M)$. When $M = S/I(G)$, we write $\pd(G)$ and $\reg(G)$ in place of $\pd(S/I(G))$ and $\reg(S/I(G))$.
\end{defn}


\section{Gap-free graphs and $\lceil 2\sqrt{n} - 2 \rceil$} \label{sec.gapfree}
The goal of this section is to provide the non-experts with an easy understanding of why $2\sqrt{n}-2$ appears naturally in the bound for $\mc(G)$. This is demonstrated by a short proof of the inequality (\ref{eq.tau}) when $G$ is a gap-free graph. We shall also construct, for any $n \ge 2$, a graph $G$ over $n$ vertices admitting $\mc(G) = \lceil 2\sqrt{n}-2\rceil$, showing that the bound for $\mc(G)$ in (\ref{eq.tau}) is sharp. The graph $G$ constructed will be gap-free and chordal.

\begin{proposition}
\label{gap-free}
If a graph $G$ on $n$ vertices is gap-free, then
\[
\mc(G) \geq \lceil 2\sqrt{n} - 2 \rceil.
\]
\end{proposition}

\begin{proof}
If $G$ is bipartite then $\mc(G) \geq n/2 \geq 2\sqrt{n} - 2$.  Thus, we can assume that $G$ is a non-bipartite graph. Since $G$ is gap-free, we can also assume that $G$ is a connected graph.

\noindent\textbf{Case 1:} $G$ contains a complete subgraph of size at least 3.  Let $q$ denote the maximum integer $q \geq 3$ for which $G$ contains a complete subgraph $K_q$.  Since $\mc(K_q) = q - 1 \geq 2\sqrt{q} - 2$, one can assume that $q < n$.  Without loss of generality, suppose that $\{x_1, \ldots, x_q\}$ is the vertex set of such a $K_q$ in $G$.

We claim that the complete subgraph $K_q$ of $G$ can be chosen such that any vertex $x_j$, for $q < j \leq n$, is connected by an edge to a vertex in this $K_q$. Indeed, suppose that this is not the case. Choose such a $K_q$ with the least number of vertices outside of $K_q$ that are not connected to any of the vertices of $K_q$. Let $x_j$, for some $q < j \leq n$, be a vertex outside of $K_q$ that is not connected to any of the vertices in $K_q$. Since $x_j$ is not an isolated vertex in $G$, there exists a vertex $x_k$, for $q < k \not=j \leq n$, such that $\{x_j, x_k\} \in E_G$. This, since $G$ is gap-free, implies that $x_k$ must be connected to at least $q - 1$ vertices of $K_q$.  In addition, it follows from the maximality of $q$ that $x_k$ must be connected to exactly $q - 1$ vertices of $K_q$.  Assume that $x_k$ is connected to $x_1, \ldots, x_{q - 1}$. Let $K_q'$ be the complete subgraph of $G$ over the vertices $\{x_1, \dots, x_{q-1}, x_k\}$.

Consider any vertex $x_l$ outside of $K_q'$ that is connected to a vertex in $K_q$. If $x_l$ is connected to any of the vertices $\{x_1, \dots, x_{q-1}\}$, then $x_l$ is still connected to that vertex in $K_q'$. If $\{x_l, x_q\} \in E_G$ and $\{x_l, x_k\} \not\in E_G$ then, by considering the pair of edges $\{x_l, x_q\}$ and $\{x_j,x_k\}$ and since $G$ is gap-free, we deduce that $\{x_l, x_j\} \in E_G$. This, again since $G$ is gap-free, implies that $x_l$ is connected to at least $q-2$ vertices among $\{x_1, \dots, x_{q-1}\}$. Thus, $x_l$ is connected to a vertex in $K_q'$ (since $q \ge 3$). Hence, the number of vertices outside $K_q'$ that are not connected to $K_q'$ is strictly less than that of $K_q$, a contradiction to the construction of $K_q$.

Now, suppose that each vertex $x_j$, for $q < j \leq n$, is connected to a vertex of $K_q$.  For $i =1, \dots, q$, let
$$W_i = \{j ~\big|~ q < j \leq n, \{x_j,x_i\} \in E_G\},$$
and set $\omega_i = |W_i|$ (note that the sets $W_i$s are not necessarily disjoint). It is easy to see that a minimal vertex cover of $G$ containing $\{x_1, \dots, x_q\} \setminus \{x_i\}$ must also contain the vertices in $W_i$. Thus, it follows that
\[
\mc(G) \geq \max\{\omega_1, \ldots, \omega_q \} + (q - 1).
\]
Therefore,
\[
\mc(G) \geq \frac{\, n - q \,}{q} + (q - 1) = \frac{\, n \,}{q} + q - 2 \geq 2\sqrt{n} - 2.
\]

\noindent\textbf{Case 2:} $G$ does not contain any complete subgraph of size at least 3. Since $G$ is not bipartite, $G$ contains an odd cycle of length $\ell \ge 5$. Since $G$ is gap-free, by considering pairs of non-adjacent edges on this cycle, we deduce that $G$ contains $C_5$ as an induced cycle. Let $x_1, \ldots, x_5$ be the vertices of this $C_5$ in $G$.

We claim that each vertex $x_i$, for $i > 5$, is connected by an edge to one of the vertices of $C_5$.  Indeed, suppose that there exists a vertex $x_i$, for some $i > 5$, that is not connected to any of the vertices of $C_5$. Since $G$ is connected, $G$ has an edge $\{x_i,x_j\}$ for some $j > 5$. Since $G$ is gap-free, either $x_1$ or $x_2$ must be connected to $x_j$. We can assume that $\{x_1,x_j\} \in E_G$. This, since $G$ has no triangle, implies that $\{x_2, x_j\} \not\in E_G$ and $\{x_5,x_j\} \not\in E_G$. For the same reason, at least one of the edges $\{x_3,x_j\}$ and $\{x_4,x_j\}$ is not in $G$. Suppose that $\{x_4, x_j\} \not\in E_G$.  We then have a gap consisting of the edges $\{x_4, x_5\}$ and $\{x_i, x_j\}$, a contradiction.

Now, for $i = 1, \dots, 5$, let $W_i = \{x_k ~\big|~ k> 5 \text{ and } \{x_i,x_k\} \in E_G\}$ and set $\omega_i = |W_i|$. Let $b_i = \omega_i+\omega_{i+2}$, where $\omega_{5+j} = \omega_j$.  Observe that a minimal vertex cover of $G$ not containing $x_i$ and $x_{i+1}$, for some $1 \le i \le 5$ must contain $W_i \cup W_{i+1}$. Since $b_1 + \cdots + b_5 \ge 2|\bigcup_{i=1}^5 W_i| = 2(n-5)$, it follows that there is $1 \leq i \leq 5$ with $b_i \geq 2n/5 -2$.  Therefore,
\[
\mc(G) \geq 2n/5 + 1 \geq 2\sqrt{n} - 2,
\]
and the result is proved.
\end{proof}

\begin{defn} \label{def.Hs}
Given $s \in \NN$, we define $H_s$ to be the graph consisting of a complete graph $K_s$, each of whose vertex is furthermore connected to an independent set of size $s-1$, and these independent sets are pairwise disjoint.
\end{defn}

\begin{center}
	\includegraphics[width=5cm, height=5cm]{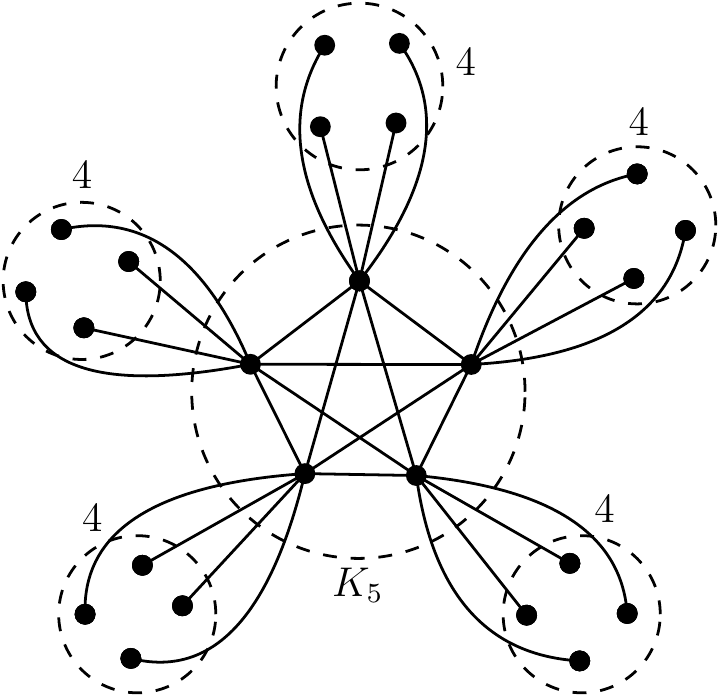}
	\captionof{figure}{$H_5$. \label{figHs}}
\end{center}
\begin{example} The graph depicted in Figure \ref{figHs} is $H_s$ for $s = 5$.
\end{example}

The following example gives a graph $G$ over $n$ vertices admitting $\mc(G) = \lceil 2\sqrt{n}-2\rceil$ for any $n \ge 2$.

\begin{example}
\label{ex.equality}
Let $a > 0$ be an integer with $a^2 \leq n < (a + 1)^2$.  If $n = a^2$ then let $G_{n} = H_a$. The first graph in Figure \ref{fig.G} is $G_{25} = H_5$. It is easy to see that $\mc(G_n) = 2(a-1) = 2\sqrt{n}-2.$

If $a^2 < n \leq a^2 + a$ then let $G_n$ be the graph obtained from $H_a$ by adding a leaf $\{x_i, x_{a^2 + i}\}$ to each vertex $x_i$, for $i=1, \dots, n-a^2$, in the complete subgraph $K_a$ of $H_a$. The second graph in Figure \ref{fig.G} is $G_{27}$.  Then $\mc(G_n) = 2a - 1$.  Since $a < \sqrt{n} \leq \sqrt{a^2 + a} < a + 1/2$, one has $\lceil 2\sqrt{n} - 2 \rceil = 2a - 1 = \mc(G_n)$.

If $a^2 + a < n < (a + 1)^2$ then let $G_n$ be the graph obtained from $G_{a^2+a}$ by adding a leaf $\{x_i, x_{a^2 + a + i}\}$ to each vertex $x_i$, for $i = 1, \dots, n-a^2-a$, in the complete subgraph $K_a$ of $G_{a^2+a}$. The third graph in Figure \ref{fig.G} is $G_{31}$. Then $\mc(G_n) = 2a$.  Since $a + 1/2 < \sqrt{a^2 + a + 1} \leq \sqrt{n} < a + 1$, one has $\lceil 2\sqrt{n} - 2 \rceil = 2a = \mc(G_n)$.

\begin{center}
	\includegraphics[width=4cm, height=4cm]{figHs.pdf} \hspace*{9pt} \includegraphics[width=4cm, height=4cm]{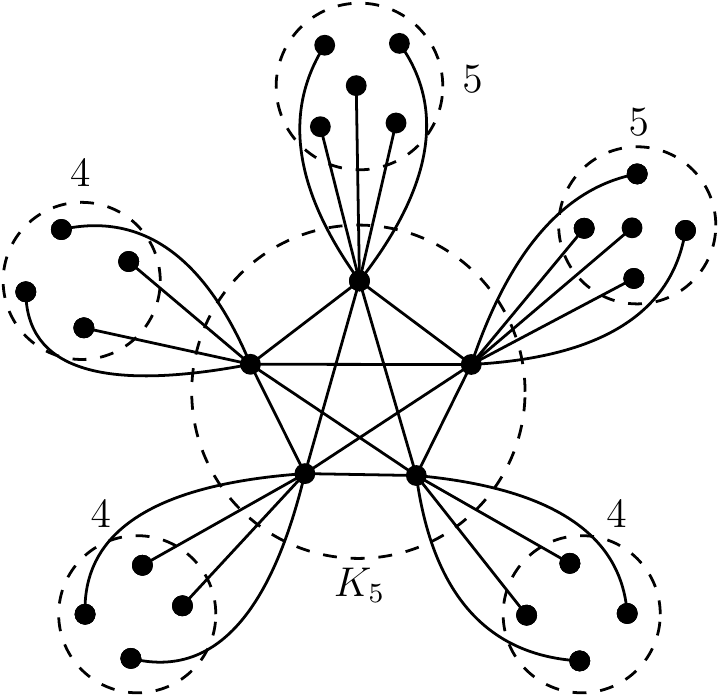} \hspace*{9pt} \includegraphics[width=4cm, height=4cm]{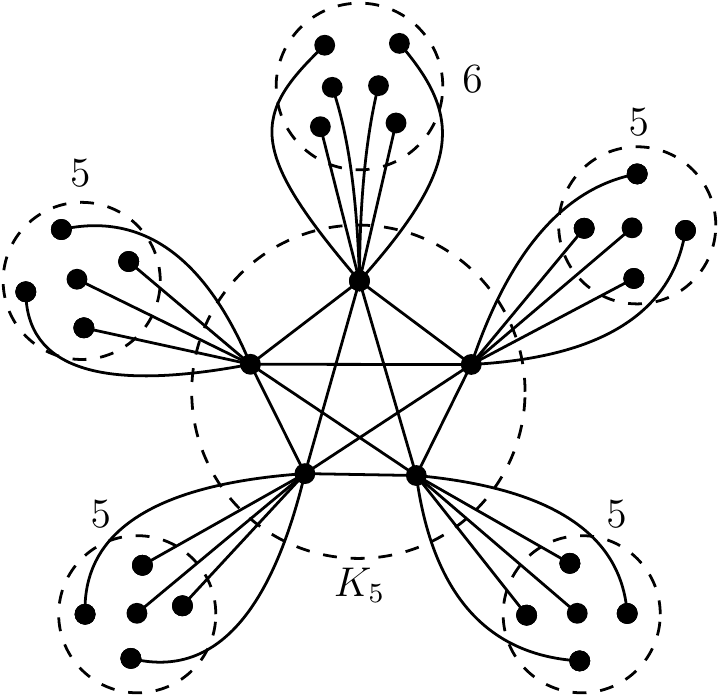}
	\captionof{figure}{Graphs with $\mc(G) = \lceil 2\sqrt{n}-2\rceil.$ \label{fig.G}}
\end{center}

Note that all graphs constructed here are chordal and gap-free.
\end{example}


\section{The bound $\mc(G) \geq 2\sqrt{n} - 2$ for an arbitrary graph} \label{sec.main}
In this section, we give a new proof for the inequality (\ref{eq.tau}). Our proof is different from that given in \cite[Theorem 2.2]{CHLN} and provides information that we could use in the next section to classify graphs for which (\ref{eq.tau}) becomes an equality. 

\begin{thm}
\label{tulane}
Let $G$ be a graph on $n$ vertices.  We have
{\em
\[
\mc(G) \geq \lceil 2\sqrt{n} - 2 \rceil.
\]
}
\end{thm}

\begin{proof} Since $\mc(G) \in \NN$, it suffices to show that $\mc(G) \ge 2\sqrt{n}-2$. Let $W$ be a minimal vertex cover of maximum size in $G$. That is, $|W| = \mc(G)$. We partition $W$ into the following two subsets
$$A' = \{w \in W ~\big|~ \deg_G(w) \ge 2\} \text{ and } B' = \{w \in W ~\big|~ \deg_G(w) = 1\}.$$
It is clear that if $A' = \emptyset$ then $G$ consists of isolated edges, and so $\mc(G) = n/2 \ge 2\sqrt{n}-2$. Thus, we shall assume that $A' \not= \emptyset$.
	
Let $B = B' \cup \{v \in A' ~|~ N_G(v) \subseteq N_G(B')\}$ and let $A = A' \setminus B$. We now have a new partition of $W$, namely, $W = A \cup B$. Note that $N_G(B') = N_G(B)$. For each vertex $v \in A$, let $M(v) = N_G(v) \setminus (W \cup N_G(B))$. Set $a = |A|$, $b = |B|$ and $n_b = |N_G(B)| \le b$.
	
Consider the maximal sets (with respect to inclusion) of the form $M(w)$ for $w \in A$, and suppose that those maximal sets are $M(w_1), \dots, M(w_t)$, for $w_1, \dots, w_t \in A$. Set $D = A \setminus \{w_1, \dots, w_t\}$ and let $d = |D|$.

\noindent\textbf{Claim 1.} For any $i = 1, \dots, t$, we have $|M(w_i)| \le b+1+d-n_b.$

\noindent\textit{Proof of Claim.} Let
$$D' =  \{w \in D ~\big|~ M(w) \not\subseteq M(w_i)\} \cup \{w \in D ~\big|~ ww_i \in E_G\}$$
and let $D'' = D \setminus D'$. Let $H$ be the induced subgraph of $G$ over $D''$ and let $U$ be a minimal vertex cover of $H$. Let
$$W' = \big[W \cup M(w_i)\cup N_G(B)\big] \setminus \big[B \cup \{w_i\} \cup (D'' \setminus U)].$$
It suffices to show that $W'$ is a minimal vertex cover of $G$, which then implies that $|W'| \le |W|$; that is,
\begin{align}
|M(w_i)| \le b+1+|D'' \setminus U|-n_b \le b+1+d-n_b. \label{eq.mvi}
\end{align}

To see that $W'$ is a vertex cover of $G$, consider any edge $e = xy \in E_G$. Since $W$ covers $G$, without loss of generality, we may assume that $x \in W$. If $x \in W'$ then $W'$ covers $e$. Assume that $x \not\in W'$. This implies that $x \in B \cup \{w_i\} \cup (D'' \setminus U)$. If $x \in B$ then $y \in N_G(B) \subseteq W'$, and so $W'$ covers $e$. If $x = w_i$ then either $y \in M(w_i) \subseteq W'$ or $y \in W \cup N_G(B)$. Furthermore, if $y \in W$ either $y$ is among the $w_j$, for $j \not= i$, or $y \in D' \subseteq W'$. Thus, in this case, $W'$ also covers $e$. If $x \in D'' \setminus U$ then, by definition, $M(x) \subseteq M(w_i)$ and $xw_i \not\in E_G$. This implies that at least one of the following happens:
\begin{enumerate}
\item $y \in M(x)$,
\item $y \in \{w_1, \dots, w_t\} \setminus \{w_i\}$,
\item $y \in N_G(B)$,
\item $y \in D'$,
\item $xy$ is an edge in $H$ (which forces $y \in U$).
\end{enumerate}
In any of these cases, we have $y \in W'$.

To see that $W'$ is a minimal vertex cover, consider any vertex cover $W'' \subseteq W'$ of $G$. Observe that $W''$ does not contain any vertex in $B \cup \{w_i\}$, so $W''$ must contain $N_G(B) \cup M(w_i)$. Also, for any $j \not= i$, $M(w_j) \not\subseteq M(w_i)$. This implies that $M(w_j) \not\subseteq W''$, which forces $w_j \in W''$ for all $j \not= i$. Furthermore, for any $w \in D'$, either $M(w) \not\subseteq W''$ or $ww_i \in E_G$. It then follows that $w \in W''$, i.e., $D' \subseteq W''$. Finally, for any vertex $u \in U$, since $U$ is a minimal vertex cover of $H$, there exists an edge $uv$ in $H$ such that $v \not\in W'$. This implies that $v \not\in W''$, and so, $u \in W''$. Hence, $W'' = W'$, and $W'$ is a minimal vertex cover of $G$. $\blacksquare$

We proceed with the proof of our theorem by considering two different cases.
	
\noindent\textbf{Case 1:} $B \not= \emptyset$. In this case, Claim 1 gives us
\begin{align}
n & = \sum_{i=1}^t |M(w_i)| +|A| + |N_G(B)| + |B| \label{eq.nb} \\
& \le t(b+1+d-n_b) + t+d+n_b+b \nonumber \\
& = 2t + t(b+d-n_b) + d+n_b+b \nonumber \\
& = 2t + (t-1)(b+d-n_b) + 2(b+d). \nonumber
\end{align}
On the other hand $\mc(G) = t+d + b$.

Observe that, since $b \ge 1$, we have $n_b \ge 1$, and so
\begin{align}
4n & \le 4\big[2t + (t-1)(b+d-1) + 2(b+d)\big] \label{eq.nb'}\\
& = 4(t+1)(b+d+1) \le (t+d+b+ 2)^2 \nonumber \\
& = (\mc(G)+2)^2. \nonumber
\end{align}
Hence, $\mc(G) \ge 2 \sqrt{n}-2$, and we are done.

\noindent\textbf{Case 2:} $B = \emptyset$. Observe that, by Claim 1, for each $i = 1, \dots, t$,
\begin{align}
|M(w_i)| \le d+1. \label{eq.mv2}
\end{align}
Observe further that if $D = \emptyset$ then it follows from (\ref{eq.mv2}) that $\mc(G) \ge t \ge n/2 \ge 2\sqrt{n}-2$. We shall assume that $D \not= \emptyset$. Since $W$ is a minimal vertex cover of $G$, it follows that $M(v) \not= \emptyset$ for all $v \in A$. Let $M(D) = \bigcup_{w \in D}M(w) \not= \emptyset$.

To prove the assertion, it suffices to show that
\begin{align}
\big|\bigcup_{i=1}^t M(w_i)\big| \le td+1. \label{eq.sum}
\end{align}
This is because (\ref{eq.sum}) then gives
\begin{align}
n \le \big|\bigcup_{i=1}^t M(w_i)| + |A| \le td+1+t+d = (t+1)(d+1), \label{eq.nb0}
\end{align}
which implies that $4n \le (t+d+2)^2 = (\mc(G)+2)^2$.

To establish (\ref{eq.sum}), we partition $\{w_1, \dots, w_t\}$ into the following two subsets
$$V_1 = \{w_i ~\big|~ M(D) \subseteq M(w_i)\} \text{ and } V_2 = \{w_i ~\big|~ M(D) \not\subseteq M(w_i)\}.$$
Consider any $w_i \in V_2$. Since $M(D) \not\subseteq M(w_i)$, there exists a vertex $x \in D$ such that $M(x) \not\subseteq M(w_i)$. Now, apply the same proof as that for Claim 1, for the set $M(w_i)$, observing that $x \in D'$ in this case, and so $|D'' \setminus U| \le d-1$. This implies that $|M(w_i)| \le d$.

Observe, finally, that if $V_1 = \emptyset$ then we have $\big|\bigcup_{i=1}^t M(w_i)\big| = \big|\bigcup_{w_i \in V_2}M(w_i)\big| \le td$, and if $V_1 \not= \emptyset$ then we have
\begin{align}
\big|\bigcup_{i=1}^t M(w_i)\big| & \le \big|\bigcup_{w_i \in V_1}M(w_i)\big| + \big|\bigcup_{w_i \in V_2}M(w_i)| \label{eq.sumb0} \\
& \le |M(D)| + (d+1-|M(D)|)|V_1| + d|V_2| \nonumber \\
& = d(|V_1|+|V_2|) + 1 - (|M(D)|-1)(|V_1|-1) \nonumber \\
& \le td+1. \nonumber
\end{align}
The result is proved.
\end{proof}

\begin{corollary} \label{cor.pd}
	Let $G$ be a graph on $n$ vertices. We have
	$$\pd(G) \ge \lceil 2\sqrt{n}-2\rceil.$$
\end{corollary}

\begin{proof} Let $I(G)^\vee$ denote the Alexander dual of the edge ideal $I(G)$ of $G$. See, for example, \cite[Chapter 5]{MS} for more details of the Alexander duality theory. By a result of Terai \cite[Theorem 2.1]{Terai}, we have
	$$\reg(I(G)^\vee) = \pd(G).$$
Observe that the minimal generators of $I(G)^\vee$ correspond to the minimal vertex covers in $G$. Since the regularity is an upper bound for the maximal generating degree, we have $\reg(I(G)^\vee) \ge \mc(G)$. Thus, $\pd(G) \ge \mc(G)$ (see also \cite{DS,K}). The assertion now follows from Theorem \ref{tulane}.
\end{proof}


\section{Classification for $\mc(G) = 2\sqrt{n}-2$ and the spectrum of $(\pd(G), \reg(G))$} \label{sec.class}
This section is devoted to our main results. We shall classify graphs $G$, when $n$ is a perfect square, for which $\mc(G)$ attains its minimum value; that is, when $\mc(G) = 2\sqrt{n}-2$. We shall also give the first nontrivial partial answer to Question \ref{question} on the spectrum of pairs of integers $(\pd(G), \reg(G))$. Recall that, for $s \in \NN$, $H_s$ is the graph defined in Definition \ref{def.Hs}.

\begin{thm}
\label{classification}
Let $G$ be a graph on $n$ vertices and suppose that $n$ is a perfect square. Then $\mc(G) = 2\sqrt{n}-2$ if and only if $G$ is either $2K_2$, $C_4$ or $H_s$, for some $s \in \NN$.
\end{thm}

\begin{proof} It is clear that if $G$ is either $2K_2$, $C_4$ or $H_s$ then $\mc(G) = \lceil 2\sqrt{n}-2\rceil$. We shall prove the other implication. Let $W$ be a minimal vertex cover of largest size. That is, $|W| = 2\sqrt{n}-2$. We shall use the same notations as in the proof of Theorem \ref{tulane}. Consider the following two possibilities.
	
\noindent\textbf{Case 1:} $B \not= \emptyset$. The proof of Theorem \ref{tulane} shows that $4n = (\mc(G)+2)^2$ only if the following conditions are satisfied:
	\begin{enumerate}
		\item[(1)] $t = b+d$ (due to (\ref{eq.nb'})),
		\item[(2)] $n_b = 1$ (due to (\ref{eq.nb'})), and
		\item[(3)] $M(w_1), \dots, M(w_t)$ are pairwise disjoint and each has exactly $b+1+d-n_b = b+d$ elements (due to (\ref{eq.nb})).
	\end{enumerate}

Condition (3), together with (\ref{eq.mvi}), implies that $|D''\setminus U| = |D|$. This happens if and only if $D'' = D$ and $U = \emptyset$. Thus, $D$ is an independent set, and for all $i = 1, \dots, t$ and $w \in D$, we have $M(w) \subseteq M(w_i)$ and $ww_i \not\in E_G$. This, together with condition (3) again, implies that either $t = 1$ or $M(w) = \emptyset$ for all $w \in D$.
	
Suppose that $t = 1$. Condition (1) then implies that $b = 1$ and $d = 0$. In this case, $G$ is either a path of length 3, i.e., $P_3$, or two disjoint edges, i.e., $2K_2$. Note that $P_3 = H_2$.
	
Suppose that $M(w) = \emptyset$ for all $w \in D$. Since $ww_i \not\in E_G$, it follows that $N_G(w) \subseteq N_G(B)$. This is a contradiction to the construction of $B$ unless $D = \emptyset$. Thus, we have $D = \emptyset$ and $A = \{w_1, \dots, w_t\}$. Let $v_b$ be the only vertex in $N_G(B)$.

Observe that if there exists an $i$ such that $w_iv_b \not\in E_G$, then let
	$$W' = [W \cup M(w_i)] \setminus \{w_i\}.$$
It can be seen that $W'$ is a minimal vertex cover of $G$. Thus, $|W'| \le |W|$, and so we must have $|M(w_i)| = 1$. That is, $b=1$ and $|M(w_j)| = 1$ for all $j = 1, \dots, t$. Therefore, $\mc(G) = n/2$. In this case, $\mc(G) = 2\sqrt{n}-2$ only if $n = 4$, and we have $t=1$ and $G = 2K_2$.
\begin{center}
\includegraphics[width=5cm, height=5cm]{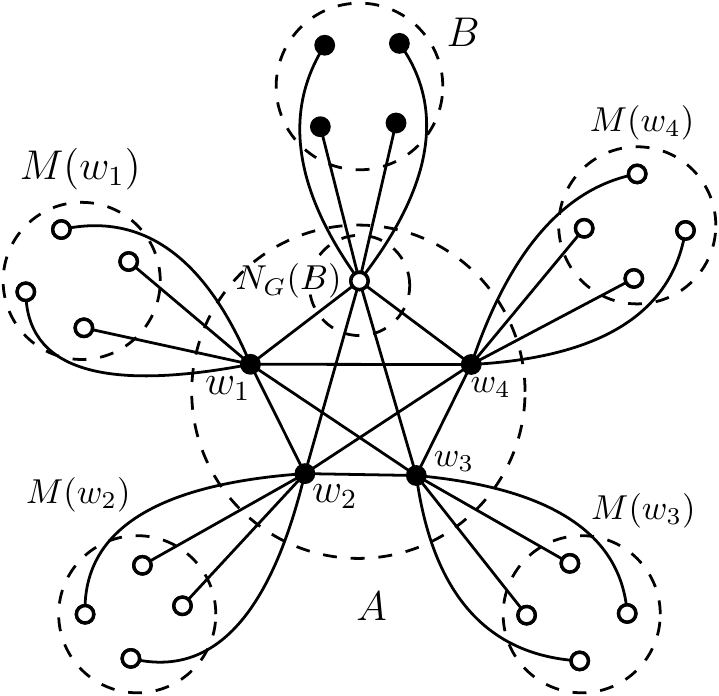}
\captionof{figure}{$G$ when $B \not= \emptyset$.\label{fig.HT}}
\end{center}

Assume that $w_iv_b \in E_G$ for all $i = 1, \dots, t$. Observe further that if there are $i \not= j$ such that $w_iw_j \not\in E_G$ then let
	$$W' = [W \cup M(w_i) \cup M(w_j) \cup N_G(B)] \setminus [B \cup \{w_i,w_j\}].$$
It can also be seen that $W'$ is a minimal vertex cover of $G$. Thus, $|W'| \le |W|$, and we get $|M(w_i)| + |M(w_j)| + 1 \le b+2$. That is, $2b \le b+1$. Therefore, $b=1$ and again $n = 4$. In this case, $G = P_3$.
	
Suppose, finally, that $w_iw_j \in E_G$ for all $i \not= j$. Clearly, we then have $G = H_{t+1}$, as depicted in Figure \ref{fig.HT} with $t = 4$.
	
\noindent\textbf{Case 2:} $B = \emptyset$. From the minimality of $W$, it follows that $M(w) \not= \emptyset$ for all $w \in A$. Suppose that $D = \emptyset$. Then (\ref{eq.mv2}) implies that $|M(w_i)| = 1$ for all $i = 1, \dots, t$. Thus, $\mc(G) \ge n/2$, and so $\mc(G) = 2\sqrt{n}-2$ only if $n=4$. Therefore, we also have that $G$ is either $P_3$ or $2K_2$.

Suppose that $D \not= \emptyset$. The proof of Theorem \ref{tulane} shows that $4n = (\mc(G)+2)^2$ only if the following conditions are satisfied:
	\begin{enumerate}
		\item[(4)] $t = d \not= 0$ (due to (\ref{eq.nb0})),
		\item[(5)] $M(w_j)$'s are all disjoint for $w_j \in V_2$, $M(w_i)$ and $M(w_j)$ are disjoint for any $w_i \in V_1$ and $j \in V_2$, and $M(w_i)$'s pairwise share exactly $M(D)$ as the set of common vertices (due to (\ref{eq.sumb0})), and
        \item[(6)] either $|M(D)| = 1$ or $|V_1| = 1$ (due to (\ref{eq.sumb0})).
	\end{enumerate}
	
Consider first the case where $|V_1| = 1$ in condition (6). Without loss of generality, we may assume that $V_1 =\{w_1\}$ and $V_2 = \{w_2, \dots, w_t\}$. In this case, condition (5) states that $M(w_1), \dots, M(w_t)$ are disjoint, $|M(w_1)| = d+1$ and $|M(w_j)| = d$ for $j \ge 2$. Particularly, for all $j \ge 2$,
\begin{align}
M(D) \cap M(w_j) = \emptyset. \label{eq.MD}
\end{align}
	
Applying (\ref{eq.mvi}) to $M(w_1)$ implies that $D$ is an independent set in $G$. Moreover, applying (\ref{eq.mvi}) to $M(w_j)$, for any $j \ge 2$, gives that $|D''| = d-1$. It follows that, for any $j \ge 2$, there is exactly one vertex $w$ in $D$ such that $M(w) \not\subseteq M(w_j)$ or $ww_j \in E_G$. This and (\ref{eq.MD}) force $d = 1$. In this case, $G$ is either a $P_3$ or a $C_4$.
	
Consider now the case where $|M(D)| = 1$. Let $v_d$ be the only vertex in $M(D)$. In this case, we have $M(w) = M(D) = \{v_d\}$ for all $w \in D$. If $V_2 \not= \emptyset$ then let $v \in V_2$. By condition (6) and applying (\ref{eq.mvi}) to $M(v)$, we deduce that $|D'' \setminus U| = d-1$. However, $M(w) \not\subseteq M(v)$ for all $w \in D$ by (\ref{eq.MD}). Thus, in applying (\ref{eq.mvi}) to estimate $M(v)$, we have $D'' = \emptyset$. This is the case only if $d = 1$. Thus, $t=d=1$, and we get to a contradiction to the fact that both $V_1$ and $V_2$ are not empty.
	
Suppose that $V_2 = \emptyset$. Condition (6) states that $M(w_1), \dots, M(w_t)$ pairwise have exactly one vertex $v_d$ in common and each is of size exactly $d+1$.
If there are $w_i$ and $w \in D$ such that $ww_i \in E_G$ then, in applying (\ref{eq.mvi}) to $M(w_i)$, we have that $D' \not= \emptyset$. That is $|D''| \le d-1$, and so $|M(w_i)| \le d$, a contradiction. Hence, $ww_i \not\in E_G$ for all $i$ and all $w \in D$. This shows that the vertices in $D$ are of degree 1. That is, $B \not= \emptyset$, a contradiction.
\end{proof}



Observe that $H_s$ is a chordal and gap-free graph. The conclusion of Theorem \ref{classification} is no longer true if $n$ is not a perfect square. In fact, for any odd integer $p \ge 3$, there exists a graph $G$ over $n = (p+1)^2/4+1$ vertices that is neither chordal nor gap-free and admits $\mc(G) = p = \lceil 2\sqrt{n}-2\rceil$. The following example depicts this scenario when $p = 5$ and $n=10$. The example for any odd $p \ge 3$ and $n = (p+1)^2/4+1$ is constructed in a similar manner.

\begin{example} \label{ex.notchordal}
Let $G$ be the following graph over $10$ vertices (as in Figure \ref{fig.10}). It is easy to see that $\mc(G) = 5 = \lceil 2\sqrt{10}-2\rceil$ (the solid black vertices form a minimal vertex cover of maximum cardinality 5). Furthermore, $G$ is neither chordal nor gap-free.
\begin{center}
\includegraphics[width=3.5cm, height=3.5cm]{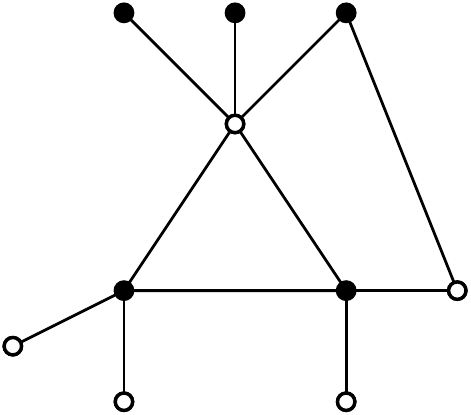}
\captionof{figure}{A graph with $\mc(G) = \lceil 2\sqrt{n}-2\rceil$ which is not chordal nor gap-free. \label{fig.10}}
\end{center}
\end{example}

The following problem, which we hope to come back to in future work, arises naturally.

\begin{problem}
	Characterize all graphs $G$ on $n$ vertices for which $\mc(G) = \lceil 2\sqrt{n}-2\rceil.$
\end{problem}

Theorem \ref{classification} furthermore gives us some initial understanding toward Question \ref{question}.

\begin{thm} \label{thm.reg}
	Let $G$ be a graph on $n$ vertices and suppose that $n$ is a perfect square. If $\mc(G) = 2\sqrt{n}-2$ then we must have either
	\begin{enumerate}
		\item $G = 2K_2$ and $(\pd(G), \reg(G)) = (2,2)$, or
		\item $G = C_4$ and $(\pd(G), \reg(G)) = (3,1)$, or
		\item $G = H_s$, for some $s \in \NN$, and $(\pd(G),\reg(G)) = (2\sqrt{n}-2,1)$.
	\end{enumerate}
\end{thm}

\begin{proof} It follows from Theorem \ref{classification} that $G$ either $2K_2$, $C_4$ or $H_s$, for some $s \in \NN$. If $G$ is $2K_2$ then $(\pd(G), \reg(G)) = (2,2).$ If $G$ is $C_4$ then $(\pd(G), \reg(G)) = (3,1)$.
	
Suppose that $G = H_s$ for some $s \in \NN$. Recall that $H_s$ is a chordal and gap-free graph. Particularly, the induced matching number of $H_s$ is 1. Thus, by \cite[Corollary 6.9]{HVT}, we have $\reg(G) =  1$.
\end{proof}

Our next main result serves as a converse to Theorem \ref{thm.reg}; that is, we identify the spectrum of $\pd(G)$ when $\reg(G) = 1$.

\begin{thm} \label{thm.spec}
	Let $n \ge 2$ be any integer. The spectrum of $\pd(G)$ for all graphs $G$, for which $\reg(G) = 1$, is precisely $[2\sqrt{n}-2, n-1] \cap \ZZ$.
\end{thm}

\begin{proof} By Theorem \ref{tulane}, we have $\pd(G) \ge 2\sqrt{n}-2$. Observe that any minimal vertex cover of $G$ needs at most $n-1$ vertices, so $\mm = (x_1, \dots, x_n)$ is not a minimal primes of $I(G)$. Furthermore, since $I(G)$ is squarefree, it has no embedded primes. This implies that $\mm$ is not an associated prime of $I(G)$. It follows that $\depth S/I(G) \ge 1$. By Auslander-Buchsbaum formula, we then have $\pd(S/I(G)) \le n-1$.

It remains to construct a graph $G$ on $n$ vertices, for any given integer $p$ such that $2\sqrt{n}-2 \le p \le n-1$, for which $\pd(G) = p$ and $\reg(G) = 1$. By considering the complete bipartite graph $K_{1,n-1}$, the assertion is clearly true for $p = n-1$. Suppose now that $p \le n-2$.

Let $s = \lceil p/2\rceil + 1$ and $T = \lfloor p/2\rfloor + 1$. It can be seen that $sT = \lfloor (p+2)^2/4\rfloor \ge n$. Note further that $s+T = p+2 \le n$. Thus, we can choose $t$ to be the largest integer such that $(s-1)t+T \le n$ (particularly, $1 \le t \le T$), and set $a = (p+2)-(s+t) = T-t$.

Let $K_s$ be the complete graph over $s$ vertices $\{x_1, \dots, x_s\}$. For each $i = 1, \dots, s$, let $W_i$ be a set of $t-1$ independent vertices such that the sets $W_i$s are pairwise disjoint and disjoint from the vertices of $K_s$. For each $i = 1, \dots, s$, connect $x_i$ to all the vertices in $W_i$. Observe further that
\begin{enumerate}
	\item $st+a = (s-1)t+T \le n$, and
	\item $st + sa = sT \ge n$.
\end{enumerate}
Thus, we can find new pairwise disjoint sets $B_1, \dots, B_s$ of independent vertices, which are also disjoint from the vertices in $K_s$ and $W_i$s, such that $|B_1| = a$ and $|B_i| \le a$, for all $2 \le i \le s$, and $\sum_{i=1}^s |B_i| = n-st$.
For each $i=1, \dots, s$, connect $x_i$ to all the vertices in $B_i$. Let $G$ be the resulting graph.

\begin{center}
	\includegraphics[width=5cm,height=5cm]{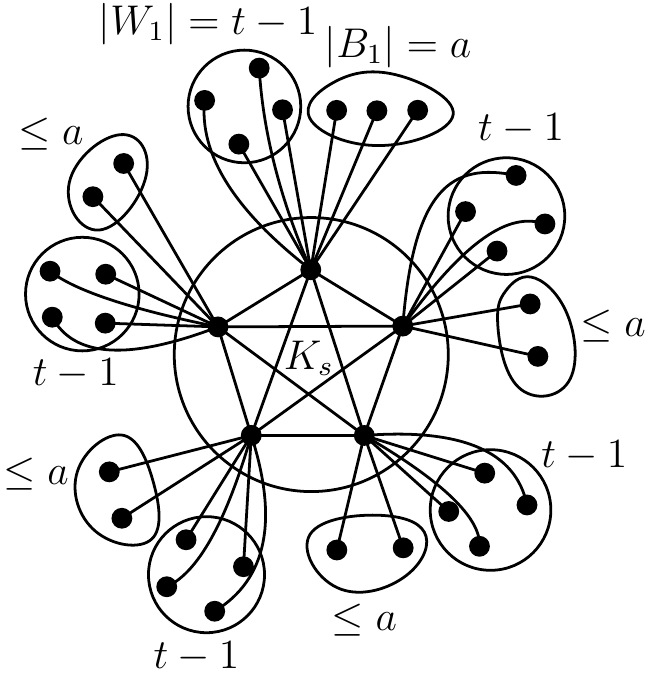}
	\captionof{figure}{A graph with $(\pd(G), \reg(G)) = (p,1).$ \label{fig.spec}}
\end{center}
It is easy to see that $G$ is a chordal and gap-free graph over $n$ vertices. It is also clear to see that $\mc(G) = (s-1)+(t-1)+a = p$. By \cite[Corollary 6.9]{HVT}, we have $\reg(G) = \nu(G) = 1$. Moreover, it follows from \cite[Theorem 3.2]{FVT} and \cite[Corollary 3.33]{MS} (see also \cite[Corollary 5.6]{DS}) that $\pd(G) = \mc(G) = p$.
\end{proof}

For simplicity of the statement of our next result, given an integer $n \ge 2$, let
$$\pdrSpec(n) = \{(p,r) ~\big|~ \text{there is a graph $G$ over $n$ vertices}: \pd(G) = p, \reg(G) = r\}.$$
Theorem \ref{thm.spec} basically states that $(p,1) \in \pdrSpec(n)$ for any integer $p$ with $2\sqrt{n}-2 \le p \le n-1$. The next corollary is an immediate consequence of Theorem \ref{thm.spec}.

\begin{corollary} \label{cor.spec}
	Let $r \le n/2$ be positive integers. We have
	$(p,r) \in \pdrSpec(n)$
	for any integer $p$ with
	$$2\sqrt{n-2(r-1)}+r-3 \le p \le n-r.$$
\end{corollary}

\begin{proof} By Theorem \ref{thm.spec}, for any integer $p'$ such that
	$$2\sqrt{n-2(r-1)}-2 \le p' \le n-2(r-1)-1,$$
	there is a graph $H$ over $n-2(r-1)$ vertices for which $\pd(H) = p'$ and $\reg(H) = 1$. Let $H'$ be the graph consisting of $r-1$ disjoint edges. It is easy to see that $\pd(H') = r-1 = \reg(H')$.
	
	Let $G$ be the disjoint union between $H$ and $H'$. Since the projective dimension and regularity are additive with respect to disjoint unions of graphs, it follows that $\pd(G) = p'+(r-1)$ and $\reg(G) = 1+(r-1) = r$. The assertion is proved by taking $p' = p - (r-1)$.
\end{proof}

Note that $\lceil 2\sqrt{n-2(r-1)}\rceil+r-3 \ge \lceil 2\sqrt{n}\rceil -2$. That is,
$$\big[\lceil 2\sqrt{n-2(r-1)}\rceil+r-3, n-r\big] \subseteq \big[\lceil 2\sqrt{n}-2\rceil, n-1\big].$$
In other words, as $\reg(G)$ gets larger, the spectrum of $\pd(G)$ appears to become smaller. Further computation does suggest that this should be true.

\begin{conjecture}
	Let $r,n \ge 2$ be arbirary integers. If $(p,r) \in \pdrSpec(n)$ then $(p,r-1) \in \pdrSpec(n)$.
\end{conjecture}



\bibliographystyle{plain}

\end{document}